\documentclass{amsart}

\usepackage{amssymb,stmaryrd,mathrsfs,dsfont,amsmath}
\usepackage{url}

\theoremstyle{plain}
\newtheorem{theorem}{Theorem}[section]
\newtheorem{lemma}[theorem]{Lemma}
\newtheorem{proposition}[theorem]{Proposition}
\newtheorem{corollary}[theorem]{Corollary}

\theoremstyle{definition}
\newtheorem{notation}[theorem]{Notation}
\newtheorem{example}[theorem]{Example}
\newtheorem{definition}[theorem]{Definition}

\theoremstyle{remark}
\newtheorem{remark}[theorem]{Remark}

\DeclareMathOperator{\rank}{rank}

\input xy
\xyoption{all}

\begin{document}

\title[Semigroups of transformations that preserve a partition]
{Regular, Unit-regular, and Idempotent elements of semigroups of transformations that preserve a partition}


\author[Mosarof Sarkar]{\bfseries Mosarof Sarkar}
\address{Department of Mathematics, Central University of South Bihar, Gaya, Bihar, India}
\email{mosarofsarkar@cusb.ac.in}

\author[Shubh N. Singh]{\bfseries Shubh N. Singh}
\address{Department of Mathematics, Central University of South Bihar, Gaya, Bihar, India}
\email{shubh@cub.ac.in}



\begin{abstract}
Let $X$ be a set and $\mathcal{T}_X$ be the full transformation semigroup on $X$. For a partition $\mathcal{P}$ of $X$, we consider semigroups $T(X, \mathcal{P}) = \{f\in \mathcal{T}_X| (\forall X_i\in \mathcal{P}) (\exists X_j \in \mathcal{P})\;X_i f \subseteq X_j\}$, $\Sigma(X, \mathcal{P}) = \{f\in T(X, \mathcal{P})|(\forall X_i \in \mathcal{P})\; Xf \cap X_i \neq \emptyset\}$, and $\Gamma(X, \mathcal{P}) = \{f\in \mathcal{T}_X|(\forall X_i\in \mathcal{P})(\exists X_j\in \mathcal{P})\; X_i f = X_j\}$. We characterize unit-regular elements of both $T(X, \mathcal{P})$ and $\Sigma(X, \mathcal{P})$ for finite $X$. We discuss set inclusion between $\Gamma(X, \mathcal{P})$ and certain semigroups of transformations preserving $\mathcal{P}$. We characterize and count regular elements and idempotents of $\Gamma(X, \mathcal{P})$. For finite $X$, we prove that every regular element of $\Gamma(X, \mathcal{P})$ is unit-regular and also calculate the size of $\Gamma(X, \mathcal{P})$.
\end{abstract}




\maketitle
\section{Introduction}

We assume that the reader is familiar with the basic terminology and facts of combinatorics and semigroup theory. Throughout the paper, let $X$ denote a set containing at least three elements, let $\mathcal{P}=\{X_i|\; i\in I\}$ denote a partition of $X$, and let $\mathcal{T}_X$ denote the full transformation semigroup on $X$. For a subset $A\subseteq X$, we denote by $Af$ the image of $A$ under $f \in \mathcal{T}_X$. We say that a map $f\in \mathcal{T}_X$ \emph{preserves} a partition $\mathcal{P}$ of $X$ if for every $X_i \in \mathcal{P}$, there exists $X_j \in \mathcal{P}$ such that $X_i f \subseteq X_j$.

\vspace{0.2cm}
In $1994$, Pei \cite{hpei-s94} introduced the subsemigroup
\[T(X, \mathcal{P}) = \{f\in \mathcal{T}_X|(\forall X_i \in \mathcal{P})(\exists X_j \in \mathcal{P})\; X_i f \subseteq X_j\}\]
of $\mathcal{T}_X$. If $\mathcal{P}$ is a trivial partition of $X$, then it is obvious that $T(X, \mathcal{P}) = \mathcal{T}_X$. Pei proved in \cite[Theorem 2.8]{hpei-s94} that $T(X, \mathcal{P})$ is the semigroup of all continuous selfmaps on $X$ endowed with the topology having $\mathcal{P}$ as a basis. Since then the semigroup $T(X, \mathcal{P})$ and its subsemigroups, for instance, $\Sigma(X, \mathcal{P}) = \{f\in T(X, \mathcal{P})|(\forall X_i \in \mathcal{P})\; Xf \cap X_i \neq \emptyset\}$ and the group of units $S(X, \mathcal{P})$ of $T(X, \mathcal{P})$ have received great attention from many semigroup theorists (see, e.g., \cite{arjo15, arjo-c04, arjo09, east-c16, east16, v-ferna-c11, v-ferna-c14, pei-s98, pei-s05, pei-c05, hpei-ac11, hpeidin-s05, shubh-c20, lsun-s07, lsun-jaa13}).

\vspace{0.2cm}
There have also been several studies focused on the regular elements and also the idempotents of semigroups of transformations that preserve a partition. It is worth mentioning that the semigroup $\mathcal{T}_X$ is regular (cf. \cite[p.33]{clifford61}). Pei \cite{pei-c05} characterized the regular elements of $T(X, \mathcal{P})$, and then concluded that the semigroup $T(X, \mathcal{P})$ is regular if and only if $\mathcal{P}$ is a trivial partition of $X$. Purisang and Rakbud \cite{puri16} characterized the regularity of $\Sigma(X, \mathcal{P})$. Dolinka et al. characterized as well as enumerated the  idempotents of $T(X, \mathcal{P})$ for a finite set $X$ in \cite{east-c16} and \cite{east16} for the uniform and non-uniform cases, respectively. The authors \cite{shubh-c20} characterized as well as enumerated the idempotents of $\Sigma(X, \mathcal{P})$ for an arbitrary set $X$ and a finite set $X$, respectively. The cardinal of semigroups of transformations that preserve a partition have also been determined (see, e.g., \cite{v-ferna12, shubh-c20, lsun-13}).


\vspace{0.2cm}
For a partition $\mathcal{P}$ of an arbitrary set $X$, let
\[\Gamma(X, \mathcal{P}) = \{f\in \mathcal{T}_X|(\forall X_i \in \mathcal{P})(\exists X_j \in \mathcal{P})\; X_i f = X_j\}.\]
It is easy to see that $\Gamma(X, \mathcal{P})$ is a subsemigroup of $T(X, \mathcal{P})$. If $X$ is a finite set and $\mathcal{P}$ is a uniform partition of $X$, a number of interesting properties of $\Gamma(X, \mathcal{P})$ have been obtained. For example, Pei proved in \cite[Theorem 4.1]{pei-s05} that $\Gamma(X, \mathcal{P})$ is exactly the semigroup of all closed selfmaps on $X$ endowed with the topology having $\mathcal{P}$ as a basis. Pei next proved in \cite[Theorem 4.1]{pei-s05} that the semigroup $\Gamma(X, \mathcal{P})$ is regular. Moreover, Pei \cite[Theorem 4.2]{pei-s05} obtained an upper bound for the rank of $\Gamma(X, \mathcal{P})$. Finally, Ara\'{u}jo and Schneider \cite[Theorem 1.1]{arjo09} computed the rank of $\Gamma(X, \mathcal{P})$.


\vspace{0.1cm}
The rest of the paper is organized as follows. In the next section, we introduce notations, definitions, and results which are used within the paper. In Section $3$, we give a characterization of the unit-regular elements of the semigroups $T(X, \mathcal{P})$ and $\Sigma(X, \mathcal{P})$ for a finite set $X$. In Section $4$, we discuss the set inclusion between $\Gamma(X, \mathcal{P})$ and certain existing semigroups of transformations that preserve the partition $\mathcal{P}$. In Section $5$, we give a characterization of the regular elements as well as the idempotents of $\Gamma(X, \mathcal{P})$. For a finite set $X$, we further prove that every regular element of $\Gamma(X, \mathcal{P})$ is unit-regular. In Section $6$, we count the number of elements, idempotents, and regular elements of $\Gamma(X, \mathcal{P})$ for a finite set $X$.

\vspace{0.1cm}
\section{Preliminaries and Notations}
The purpose of this section is to introduce notations, definitions, and results that we will use within the paper. We refer the reader to the standard books \cite{bru10, howie95} for additional information from combinatorics and semigroup theory, respectively.

\vspace{0.1cm}
Unless otherwise specified, we will use capital letter to denote nonempty subset, calligraphic letter to denote collection of subsets, and small letter to denote set-element, map, or positive integer. The letter $I$ will be reserved for an arbitrary indexing set. We denote by $\mathbb{N}$ the set of all positive integers. For $m \in \mathbb{N}$, the symbol $I_m$ denote the subset $\{1,\ldots, m\}$. The \emph{size} or \emph{cardinality} of a set $A$ is the number of elements in $A$, and it is denoted by $|A|$. A finite set of size $n \in \mathbb{N}$ is called an \emph{$n$-element set} or \emph{$n$-set}. We write $A\setminus B$ to denote the set of all elements $x \in A$ such that $x\notin B$. We denote by $\binom{n}{r}$ the number of $r$-subsets of an $n$-set.

\vspace{0.2cm}
Let $X$ be a nonempty set, and let $m, k \in \mathbb{N}$ such that $m \ge k$. A \emph{partition} of $X$ is a collection of nonempty disjoint subsets, called \emph{blocks}, whose union is $X$. A partition is called \emph{trivial} if it has only singleton blocks or a single block. A partition $\mathcal{P}$ is called \emph{uniform} if all the blocks of $\mathcal{P}$ have the same size; otherwise, $\mathcal{P}$ is called \emph{non-uniform}. A partition is said to be \emph{$m$-partition} if it has exactly $m$ blocks. If an $m$-partition has exactly $k$ distinct sizes of blocks, we say that the partition is an \emph{$(m, k)$-partition}. A \emph{$k$-subpartition} of an $m$-partition $\mathcal{P}$ is a subcollection of $\mathcal{P}$ containing $k$ blocks. The number of partitions of an $m$-element set into $k$ blocks is denoted by $S(m, k)$, and is called the \emph{Stirling number of the second kind} (cf. \cite[Theorem 8.2.5]{bru10}). It is also known that the number of surjective maps from an $m$-element set onto a $k$-element set is $k!S(m, k)$.

\vspace{0.2cm}
A \emph{selfmap} on a set $A$ is a map from $A$ to itself. The composition of maps will be denoted by juxtaposition. Let $f, g \in \mathcal{T}_X$ and $x\in X$. We will write $xf$ to denote the image of $x$ under $f$, and compose maps from left to right: $x(fg) = (xf)g$. The domain, codomain, and range of a map $\alpha$ will be denoted by $\mbox{dom}(\alpha)$, $\mbox{codom}(\alpha)$, and $\mbox{ran}(\alpha)$, respectively. The \emph{rank} of a map $\alpha$, denoted by $\mbox{rank}(\alpha)$, is the cardinality of $\mbox{ran}(\alpha)$. The pre-image of $B\subseteq X$ under $f$ is denoted by $Bf^{-1} = \{x\in X\;|\; xf\in B\}$. If $A, B \subseteq X$ such that $Af \subseteq B$, then there is a map $h \colon A \to B$ defined by $xh = xf$ for all $x\in A$ and, in this case, we say that the map $h$ \emph{is induced by} $f$.

\vspace{0.2cm}
Let $S$ be a semigroup with identity. An element $a\in S$ is said to be \emph{regular} if there exists $b \in S$ such that $a = aba$; otherwise, $a$ is called \emph{irregular}. If each element of $S$ is regular, then we say that the semigroup $S$ is \emph{regular}. The set of all regular elements of $S$ will be denoted by $\mbox{Reg}(S)$. If $A \subseteq S$, then we write $\mbox{Reg}(A)$ for the set $A \cap \mbox{Reg}(S)$. An element $a$ of $S$ is said to be \emph{unit-regular} if there exists a unit $u \in S$ such that $a = aua$. Note that every unit-regular element of $S$ is contained in $\mbox{Reg}(S)$.  If each element of $S$ is unit-regular, then we say that the semigroup $S$ is \emph{unit-regular}. An element $a \in S$ is called an \emph{idempotent} if $a^2 = a$. The set of all idempotents of $S$ is denoted by $E(S)$. Note that $E(S) \subseteq \mbox{Reg}(S)$. If $A \subseteq S$, then we write $E(A)$ for the set $A \cap E(S)$. It is worth mentioning that $f\in \mathcal{T}_X$ is an idempotent if and only if $f$ acts as the identity map on its image set (cf. \cite[p.6]{clifford61}).

\vspace{0.1cm}

\section{Unit-regular elements of $T(X, \mathcal{P})$}

In this section, we give a characterization of the unit-regular elements of the semigroup $T(X, \mathcal{P})$ for a finite set $X$. Using this, we also characterize the unit-regular elements of the semigroup $\Sigma(X, \mathcal{P})$ for a finite set $X$. Recall that an element $a$ of a semigroup with identity is unit-regular if there exists a unit $u$ such that $a = aua$. The definition of unit-regular elements of a semigroup with identity is a natural one. The notion of unit-regular elements was first appeared in the context of rings \cite{ehrlich68}. There have been a number of interesting works on the unit-regular elements and unit-regular semigroups (see, e.g., \cite{t-blyth83, y-chaiya19, schen74, alarcao80, j-fount02, hickey97, ytira79}).

\vspace{0.2cm}
It is worth mentioning that the semigroup $\mathcal{T}_X$ is unit-regular if and only if $X$ is a finite set (cf. \cite[Proposition 5]{alarcao80}). Thus, if $\mathcal{P}$ is a trivial partition of a finite set $X$ then $T(X, \mathcal{P})$ is unit-regular. Pei \cite{pei-c05} gave a characterization of the regular elements of $T(X, \mathcal{P})$. However, there can be a regular element in $T(X, \mathcal{P})$ for a nontrivial partition $\mathcal{P}$ of $X$ which is not unit-regular as shown in the following example.

\begin{example}
Let $\mathcal{P}=\{\{1\},\{2,3\}\}$ be a partition of $X=\{1,2,3\}$. Define a selfmap $f\colon X\to X$ by
\[f =
\left(
  \begin{array}{ccc}
    1 & 2 & 3 \\
    2 & 1 & 1 \\
  \end{array}
\right).\]

It is clear that $f\in T(X,\mathcal{P})$. Since $f = fff$, it follows that $f$ is a regular element of $T(X, \mathcal{P})$. We can also observe that there are only two selfmaps in  $T(X,\mathcal{P})$, namely $f$ and $g$ such that $f= fff$ and $f= fgf$, where $g\colon X\to X$ is defined by
\[g = \left(
  \begin{array}{ccc}
    1 & 2 & 3 \\
    3 & 1 & 1 \\
  \end{array}\right).\]
Note that both selfmaps $f$ and $g$ on $X$ are not bijective. Hence $f$ is not a unit-regular element of $T(X,\mathcal{P})$.
\end{example}

Let us now recall a definition from \cite{puri16}.

\begin{definition} (cf. \cite{puri16})
Let $\mathcal{P}=\{X_i|\;i\in I\}$ be a partition of a set $X$, and let $f \in T(X,\mathcal{P})$. Then the \emph{character} of $f$, denoted by $\chi^{(f)}$, is a selfmap $\chi^{(f)}\colon I \to I$ defined by  \[i\chi^{(f)}=j \mbox{ whenever } X_i f\subseteq X_j.\]
\end{definition}

\vspace{0.05cm}
If $X$ is a finite set, the selfmap  $\chi^{(f)}$ has also been discussed and denoted by $\overline{f}$ (see \cite{arjo15, east-c16, east16}).

\vspace{0.1cm}
The next remark is simple.
\begin{remark}\label{gamma-char}
Let $\mathcal{P} = \{X_i|\; i\in I\}$ be a partition of a set $X$, and let $f\in \Gamma(X, \mathcal{P})$. Then \[i\chi^{(f)} = j \mbox{ if and only if } X_i f = X_j.\]
\end{remark}

\vspace{0.2cm}

The following theorem characterizes the unit-regular elements of the semigroup $T(X,\mathcal{P})$ for a finite set $X$.

\begin{theorem}\label{unit-reg-txp}
Let $\mathcal{P} = \{X_i|\; i\in I_m\}$ be an $m$-partition of a finite set $X$, and let $f\in T(X,\mathcal{P})$. Then $f$ is unit-regular if and only if for each $i\in I_m\chi^{(f)}$ there exists $j \in I_m$ such that $|X_i| = |X_j|$ and $X_i \cap Xf = X_j f$.
\end{theorem}

\begin{proof}
Suppose first that $f\in T(X,\mathcal{P})$ is unit-regular. Then there exists $g\in S(X,\mathcal{P})$ such that $f= fgf$. Let $i\in I_m\chi^{(f)}$.
Since $g\in S(X,\mathcal{P})$, there exists $j\in I_m$ such that $X_i g \subseteq X_j$ and so $|X_i| = |X_j|$ by \cite[Lemma 3.6 (ii)]{shubh-c20}.
We now verify that $X_i \cap Xf = X_j f$.

\vspace{0.0cm}
Since $i\in I_m\chi^{(f)}$, there exists $k\in I_m$ such that $X_k f \subseteq X_i$. Recall that $X_i g \subseteq X_j$ and $f= fgf$. We then obtain
\[X_k f = (X_k f)gf \subseteq X_i(gf) = (X_ig)f \subseteq X_j f.\] It concludes that $X_j f \subseteq X_i$ and subsequently $X_j f \subseteq X_i \cap Xf$. On the other hand, let $y \in X_i \cap Xf$. Then there exists $x\in X$ such that $xf = y$. Since $y \in X_i$ and $X_i g \subseteq X_j$, we have $yg \in X_j$. Since $f = fgf$, we then obtain \[y = xf = (xf)gf = (yg)f \in X_j f.\] Thus $X_i \cap Xf \subseteq X_j f$ and consequently $X_i \cap Xf = X_j f$. Since $i\in I_m\chi^{(f)}$ is an arbitrary element, the necessary part of the proof is complete.

\vspace{0.1cm}
Conversely, assume that the condition holds and we need to find a bijection $g \in S(X, \mathcal{P})$ such that $f = fgf$.

\vspace{0.1cm}
If $i\in I_m\chi^{(f)}$, by hypothesis, there exists $j\in I_m$ such that $|X_i| = |X_j|$ and $X_i \cap Xf = X_j f$. For each $x'\in X_jf$, we arbitrarily fix an element $x''$ in $\{x\in X_j\;|\; xf = x'\}$. Note that $x' \in X_i$, $x''\in X_{j}$, and $x''f=x'$. Since $|X_i|=|X_j|$, we can choose a bijective map $g_i\colon X_i\to X_j$ such that $x'g_i=x''$.

\vspace{0.1cm}
If $i\in I_m\setminus I_m\chi^{(f)}$, then we consider the following two possibilities separately.

\vspace{0.1cm}
\noindent \textbf{Case} ($X_i = \mbox{codom}(g_j)$ for some $j\in I_m\chi^{(f)}$): Then we choose a bijective map $h_i$ from $X_i$ onto a block $X_k$, where $k\in I_m\chi^{(f)}$ and $X_k$ is not the image of some previously defined bijective map.

\vspace{0.1cm}
\noindent \textbf{Case} ($X_i \neq \mbox{codom}(g_j)$ for all $j\in I_m\chi^{(f)}$): Then we choose the identity map on the block $X_i$.


Using these bijective maps on blocks of $\mathcal{P}$, we now define a selfmap $g\colon X\to X$ by setting
\begin{equation*}
xg=
\begin{cases}
xg_i, & \text{ if $x\in X_i$ where $i\in I_m\chi^{(f)}$};\\
xh_i, & \text{ if $x\in X_{i}$ where $i\in I_m\setminus I_m\chi^{(f)}$ and $X_{i}=\mbox{codom}(g_j)$ for some $j\in I_m\chi^{(f)}$};\\
x, & \text{otherwise}.
\end{cases}
\end{equation*}
Observe that $g\in S(X,\mathcal{P})$. We finally show that $f = fgf$. Let $x\in X$. Then $xf \in X_i$ for some $i\in I_m\chi^{(f)}$. Set $xf = y$ and $y g_i = z$. Then $zf = y$ by definition of bijection $g_i$. We therefore obtain \[x(fgf) = y(gf) = zf = y = xf.\] Since $x\in X$ is an arbitrary element, we have $f = fgf$. Hence $f$ is a unit-regular of the semigroup $T(X, \mathcal{P})$. This completes the proof.
\end{proof}

\vspace{0.1cm}

If $X$ is a finite set, by using Theorem \ref{unit-reg-txp} and Corollary 3.5 of \cite{shubh-c20}, we prove the following proposition which characterizes the unit-regular elements of the semigroup $\Sigma(X,\mathcal{P})$.
\begin{proposition}
Let $\mathcal{P} = \{X_i|\; i\in I_m\}$ be an $m$-partition of a finite set $X$, and let $f\in \Sigma(X,\mathcal{P})$. Then $f$ is unit-regular if and only if $|X_i| = |X_j|$ whenever $i\chi^{(f)} = j$.
\end{proposition}

\begin{proof}
Suppose first that $f\in \Sigma(X,\mathcal{P})$ is unit-regular. Since $\Sigma(X,\mathcal{P})\subseteq T(X,\mathcal{P})$, by Theorem \ref{unit-reg-txp},  for each $j\in I_m\chi^{(f)}$ there exists $i\in I_m$ such that $|X_i| = |X_j|$ and $Xf \cap X_j = X_i f$. It simply concludes that $i \chi^{(f)} = j$. Since $\mathcal{P}$ is an $m$-partition, by \cite[Corollary 3.5]{shubh-c20}, the selfmap $\chi^{(f)}$ is bijective on $I_m$ and hence $i\chi^{(f)} = j$.

\vspace{0.1cm}

Conversely, suppose that $|X_i| = |X_j|$ whenever $i\chi^{(f)} = j$. By definition of $\chi^{(f)}$, we have $X_i f \subseteq X_j$. Since $\mathcal{P}$ in an $m$-partition of finite set $X$ and $f\in \Sigma(X,\mathcal{P})$, by \cite[Corollary 3.5]{shubh-c20}, the selfmap $\chi^{(f)}$ is bijective on $I_m$. It follows that $X_i f = X_j \cap Xf$. Hence $f$ is unit-regular by Theorem \ref{unit-reg-txp}.
\end{proof}

\vspace{0.1cm}
\section{Set Inclusion}
In this section, we discuss the set inclusion between $\Gamma(X, \mathcal{P})$ and certain known semigroups of transformations that preserve the partition $\mathcal{P}$. If $X$ is a finite set, we also observe that the intersection of the semigroups $\Gamma(X, \mathcal{P})$ and $\Sigma(X, \mathcal{P})$ is exactly the group of units $S(X, \mathcal{P})$ of the semigroup $T(X, \mathcal{P})$. We begin with the following interesting theorem, whose proof is similar to the proof of Theorem 4.1 in \cite{pei-s05}.

\begin{theorem}
Let $\mathcal{P}$ be a partition of an arbitrary set $X$. Then $\Gamma(X, \mathcal{P})$ is the semigroup of all closed selfmaps on $X$ endowed with topology having $\mathcal{P}$ as a basis.
\end{theorem}

\vspace{0.1cm}
Note that $\Gamma(X, \mathcal{P}) \subseteq T(X, \mathcal{P}) \subseteq \mathcal{T}_X$. Moreover, we have the following obvious remark.
\newpage
\begin{remark}
Let $\mathcal{P}$ be a trivial partition of a set $X$.
\begin{enumerate}
\item[\rm(i)] If $\mathcal{P}$ has singleton blocks, then $\Gamma(X,\mathcal{P})= \mathcal{T}_X$.
\item[\rm(ii)] If $\mathcal{P}$ has a single block, then $\Gamma(X,\mathcal{P})$ is the semigroup of all surjective transformations on $X$.
\item[\rm(iii)] If $X$ is an $n$-element set and $\mathcal{P}$ has a single block, then $\Gamma(X,\mathcal{P})$ is the symmetric group on $X$.
\end{enumerate}
\end{remark}

\vspace{0.05cm}
The next proposition provides a necessary condition for a map in $T(X,\mathcal{P})$ to be in $\Gamma(X,\mathcal{P})$.

\begin{proposition}\label{image-size-equal}
Let $\mathcal{P} = \{X_i|\; i\in I\}$ be a partition of a set $X$, and let $f\in T(X,\mathcal{P})$. If $f\in \Gamma(X,\mathcal{P})$, then there exists $X_i \in\mathcal{P}$ such that $|X_i| = |X_if|$.
\end{proposition}

\begin{proof}
If $f\in \Gamma(X,\mathcal{P})$, then $X_if\in \mathcal{P}$ for all $X_i \in \mathcal{P}$. It follows that $|X_i f| \le |X_i|$ for each $X_i\in \mathcal{P}$. Let $X_k \in \mathcal{P}$ such that $|X_k| \le |X_j|$ for all  $X_j \in \mathcal{P}$. Since $X_k$ is a block of the smallest size, we then have $|X_k|=|X_kf|$. This completes the proof.
\end{proof}

\begin{lemma}\label{s-sub-gama}
If $\mathcal{P}$ is a partition of a set $X$, then $S(X,\mathcal{P})\subseteq  \Gamma(X,\mathcal{P})$.
\end{lemma}

\begin{proof}
Let $f\in S(X,\mathcal{P})$. Then, by \cite[Lemma 3.6(i)]{shubh-c20}, we have $X_i f \in \mathcal{P}$ for all $X_i \in \mathcal{P}$. It simply concludes that $f\in \Gamma(X,\mathcal{P})$. Hence the result follows.
\end{proof}

There is now a natural curiosity that whether or not every map of $\Gamma(X,\mathcal{P})$ belongs to $S(X,\mathcal{P})$. We give a partial answer affirmatively to the question in the following proposition.

\begin{proposition}\label{Gamma-bij-char}
Let $\mathcal{P} = \{X_i|\; i\in I_m\}$ be an $m$-partition of a finite set $X$, and let $f\in \Gamma(X, \mathcal{P})$. If $\chi^{(f)}$ is a bijective map, then $f\in S(X, \mathcal{P})$.
\end{proposition}

\begin{proof}
By hypothesis, it suffices to show that $f$ is surjective. Let $y\in X$. Then there exists $j\in I_m$ such that $y\in X_j$. Since the map $\chi^{(f)}$ is bijective, there exists $i\in I_m$ such that $i \chi^{(f)} = j$. Then, by Remark \ref{gamma-char}, it follows that $X_i f = X_j$ and so $y\in X_i f$. Hence $f$ is surjective. This completes the proof.
\end{proof}

If $X$ is an arbitrary set, the following example shows that the above Proposition \ref{Gamma-bij-char} need not be true.

\begin{example}\label{counter-01}
Let $X=\mathbb{Z}$. Consider the partition $\mathcal{P}=\{X_1, X_2\}$ of $X$, where $X_1$ and $X_2$ denote the set of nonnegative integers and the set of negative integers, respectively. Define a selfmap $f\colon X\to X$ by $xf=\lfloor \frac{x}{2} \rfloor$,  where $\lfloor x \rfloor$ denotes the greatest integer less than or equal to $x$.

\vspace{0.1cm}
Clearly $f \in T(X, \mathcal{P})$. Note that  $0f=1f=0$. It follows that $f$ is not injective and so $f\notin S(X,\mathcal{P})$. Observe that the map $\chi^{(f)}$ is bijective and so $f\in \Sigma(X, \mathcal{P})$ by \cite[Corollary 3.5]{shubh-c20}. One can also verify in a routine manner that $f\in \Gamma(X,\mathcal{P})$.
\end{example}

\vspace{0.1cm}

The following corollary is an immediate consequence of Proposition \ref{Gamma-bij-char} and \cite[Lemma 3.6(ii)]{shubh-c20}.
\begin{corollary}
Let $\mathcal{P} = \{X_i|\; i\in I_m\}$ be an $m$-partition of a finite set $X$, and let $f\in \Gamma(X, \mathcal{P})$ such that the map $\chi^{(f)}$ is bijective. If $X_i f = X_j$, then $|X_i| = |X_j|$.
\end{corollary}

\vspace{0.1cm}

If $X$ is a finite set, we now prove that the intersection of the semigroups $\Gamma(X, \mathcal{P})$ and $\Sigma(X, \mathcal{P})$ is exactly the group of units of the semigroup $T(X, \mathcal{P})$ in the following proposition.
\begin{proposition}\label{gam-sig-perm}
Let $\mathcal{P} = \{X_i|\; i\in I_m\}$ be an $m$-partition of a finite set $X$. Then \[\Gamma(X, \mathcal{P}) \cap \Sigma(X, \mathcal{P}) = S(X, \mathcal{P}).\]
\end{proposition}

\begin{proof}
Let $f\in S(X, \mathcal{P})$. Then we have $f \in \Gamma(X, \mathcal{P})$ by Lemma \ref{s-sub-gama}. Moreover, since $f$ is bijective, it follows that $Xf = X$. Therefore, $Xf \cap X_i \neq \emptyset$ for all $X_i \in \mathcal{P}$. Hence $f\in \Sigma(X, \mathcal{P})$ and consequently $S(X, \mathcal{P}) \subseteq \Gamma(X, \mathcal{P}) \cap \Sigma(X, \mathcal{P})$.

\vspace{0.1cm}
On the other hand, let $f \in \Gamma(X, \mathcal{P}) \cap \Sigma(X, \mathcal{P})$. Since $f \in \Sigma(X, \mathcal{P})$, the selfmap $\chi^{(f)}$ on $I_m$ is bijective by \cite[Corollary 3.5]{shubh-c20}. Moreover, since $f \in \Gamma(X, \mathcal{P})$, we have $f \in S(X, \mathcal{P})$ by Proposition \ref{Gamma-bij-char}. Hence $\Gamma(X, \mathcal{P}) \cap \Sigma(X, \mathcal{P}) \subseteq S(X, \mathcal{P})$. This completes the proof.
\end{proof}

If $X$ is an arbitrary set, the above Proposition \ref{gam-sig-perm} is, in general, not true as one can quickly see from Example \ref{counter-01}.

\section{Regular elements and Idempotents of $\Gamma(X, \mathcal{P})$}
In this section, we give a characterization of the regular elements as well as the idempotents of the semigroup $\Gamma(X, \mathcal{P})$ in the respective subsections. Let us first recall a definition and lemma from \cite{shubh-c20}.

\begin{definition}\cite[Definition 5.1]{shubh-c20}
Let $\mathcal{P}$ be a partition of a set $X$. A \emph{block map} is a map whose domain and codomain are the blocks of $\mathcal{P}$.
\end{definition}

\begin{lemma}\cite[Lemma 5.2]{shubh-c20}\label{fam-func}
Let $\mathcal{P} = \{X_i|\;i\in I\}$ be a partition of a set $X$, and let $f \in \mathcal{T}_X$. Then $f\in T(X,\mathcal{P})$ if and only if there exists a unique indexed family $B(f, I)$ of block maps induced by $f$, where \[B(f, I) = \{f_i |\;  f_i \mbox{ is induced by }f\mbox{and } \mbox{dom}(f_i) = X_i \mbox{ for each }i\in I\}.\]
\end{lemma}

\vspace{0.1cm}
The next remark is simple.

\begin{remark}\label{gamma-block-surj}
If $f\in \Gamma(X, \mathcal{P})$, then each block map of the family $B(f, I)$ is surjective.
\end{remark}

\vspace{0.0cm}

\subsection{Regular elements}

In this subsection, we first give a characterization of the regular elements of the semigroup $\Gamma(X, \mathcal{P})$. We next observe that if $\mathcal{P}$ is a partition of a finite set $X$ containing at most two blocks, then $\Gamma(X, \mathcal{P})$ is a regular semigroup. For a finite set $X$, we also prove that every regular element of the semigroup $\Gamma(X, \mathcal{P})$ is unit-regular.

\vspace{0.2cm}
If $\mathcal{P}$ is a uniform partition of a finite set $X$, we know that the semigroup $\Gamma(X, \mathcal{P})$ is regular (cf. \cite[Theorem 4.1]{pei-s05}). However, for a non-uniform partition $\mathcal{P}$ of $X$, the semigroup $\Gamma(X, \mathcal{P})$ is, in general, not regular as the following example shows.

\begin{example}
Let $\mathcal{P} = \big\{\{1\}, \{2\}, \{3,4\}\big\}$ be a non-uniform partition of $X = \{1,2,3,4\}$. Define a selfmap $f\colon X \to X$ by
\[f = \left(
    \begin{array}{cccc}
      1 & 2 & 3 & 4 \\
      1 & 1 & 2 & 2 \\
    \end{array}\right).\]

It is clear that $f \in \Gamma(X, \mathcal{P})$. We now show that $f$ is irregular. Assume, to the contrary, that there exists $g\in \Gamma(X, \mathcal{P})$ such that $f = fgf$. Since $3f = 2$ and $f= fgf$, we then obtain
\[3(fgf) = 2\Longrightarrow (2g)f = 2.\]
It concludes that $2g = 3$ or $2g = 4$. But, in either case, the image of $\{2\}\in \mathcal{P}$ under $g$ does not belong to $\mathcal{P}$. This contradicts our assumption that $g\in \Gamma(X, \mathcal{P})$. Hence $f$ is an irregular element of $\Gamma(X, \mathcal{P})$.
\end{example}

\vspace{0.1cm}

The following theorem characterizes the regular elements of the semigroup $\Gamma(X, \mathcal{P})$ in terms of certain  block maps.
\begin{theorem}\label{Gamma-regulr}
Let $\mathcal{P} = \{X_i|\; i \in I\}$ be a partition of an arbitrary set $X$, and let $f\in \Gamma(X,\mathcal{P})$. Then $f$ is regular if and only if for every $j\in I\chi^{(f)}$ there exists $i\in I$ such that the block map $f_i\colon X_i\to X_j$ of the family $B(f, I)$ is injective.
\end{theorem}

\begin{proof}
Suppose first that $f\in \Gamma(X,\mathcal{P})$ is regular. Let $j \in I\chi^{(f)}$. Then there exists $i\in I$ such that $i\chi^{(f)} = j$ and so $X_if = X_j$ by Remark \ref{gamma-char}. We then have a block map $f_i\colon X_i \to X_j$ in $B(f, I)$. If $f_i$ is injective, then we are done. Otherwise, suppose that $f_i$ is not injective.

\vspace{0.1cm}
Since $f\in \Gamma(X,\mathcal{P})$ is regular, there exists $g\in \Gamma(X,\mathcal{P})$ such that $f = fgf$. Moreover, since $j\in I\chi^{(f)} \subseteq I$, there exists $l\in I\chi^{(g)}$ such that $j\chi^{(g)} = l$ and so
$X_jg = X_l$ by Remark \ref{gamma-char}. We then have a block map $g_j \colon X_j \to X_l$ in $B(g, I)$.

\vspace{0.1cm}

Recall that $X_if=X_j$, $X_jg = X_l$, and $f= fgf$. We then obtain
\[X_j = X_i(fgf)=(X_i f)gf= (X_j g)f= X_l f,\]
and subsequently there is a block map $f_l \colon X_l \to X_j$ in $B(f, I)$. We now claim that the block map $f_l$ is injective.

\vspace{0.1cm}
Assume, to the contrary, that there exist two distinct elements $x, y \in X_l$ such that $xf=yf$. By Remark \ref{gamma-block-surj}, we know that the block map $g_j \colon X_j \to X_l$ is surjective. Therefore, for distinct elements $x$ and $y$ of $X_l$, there exist two distinct elements $x', y' \in X_j$ such that $x'g=x$ and $y'g=y$.
Write $xf= z \in X_j$. Since $x'$ and $y'$ are distinct elements of $X_j$, the element $z$ may be equal to at most one element of $\{x', y'\}$. We can then verify in a routine manner that $f \neq fgf$ which is a contradiction, and so the necessity follows.

\vspace{0.1cm}
Conversely, suppose that the condition holds and we need to find a map $g \in \Gamma(X, \mathcal{P})$ such that $f =fgf$. By Remark \ref{gamma-block-surj}, we know that every block map of $B(f, I)$ is surjective. Therefore, by hypothesis, for every $j\in I\chi^{(f)}$, there exists $i\in I$ such that the block map $f_i\colon X_i\to X_j$ is bijective. Denote by $h_j$ the inverse map of the bijective map $f_i\colon X_i \to X_j$. Note that the inverse map $h_j\colon X_j\to X_i$ is also bijective. Define a map $g\colon X\to X$ by
\begin{align*}
xg=
\begin{cases}
xh_j, & x\in X_j \mbox{ where } j\in I\chi^{(f)};\\
x, & \mbox{otherwise}.
\end{cases}
\end{align*}
It is clear that $g\in \Gamma(X,\mathcal{P})$. One can also verify in a routine manner that $f=fgf$ and so $f$ is a regular element of $\Gamma(X,\mathcal{P})$. This completes the proof.
\end{proof}

\begin{corollary}\label{unif-reg}
If $\mathcal{P} = \{X_i|\; i\in I_m\}$ is a uniform $m$-partition of a finite set $X$, then the semigroup $\Gamma(X,\mathcal{P})$ is regular.
\end{corollary}

\begin{proof}
Let $f\in \Gamma(X,\mathcal{P})$, and let $j\in I_m\chi^{(f)}$. Then there exists $i \in I_m$ such that $i\chi^{(f)} = j$ and so $X_if=X_j$ by Remark \ref{gamma-char}. We then have the surjective block map $f_i\colon X_i\to X_j$ in $B(f, I_m)$. By hypothesis, we know that $|X_i| = |X_j|$. It follows that $f_i$ is injective (cf. \cite[Proposition 1.1.3]{gan-maz09}). Then $f$ is a regular element of $\Gamma(X,\mathcal{P})$ by Theorem \ref{Gamma-regulr}. Hence, since $f\in \Gamma(X,\mathcal{P})$ is arbitrary, the semigroup $\Gamma(X,\mathcal{P})$ is regular.
\end{proof}

If the size of a partition $\mathcal{P}$ of a finite set $X$ is at most two, the following proposition proves that the semigroup $\Gamma(X, \mathcal{P})$ is regular.

\begin{proposition}
If $\mathcal{P}$ is a partition of a finite set $X$ such that $|\mathcal{P}| \le 2$, then the semigroup $\Gamma(X, \mathcal{P})$ is regular.
\end{proposition}

\begin{proof}
If $|\mathcal{P}|=1$, then it is clear that $\Gamma(X,\mathcal{P})$ is the symmetric group on $X$ and so $\Gamma(X,\mathcal{P})$ is regular. Therefore, suppose that $\mathcal{P}=\{X_1,X_2\}$. Without loss of generality, assume that $|X_1|\leq |X_2|$. If $|X_1|= |X_2|$, then $\mathcal{P}$ is a uniform partition and so the semigroup $\Gamma(X,\mathcal{P})$ is regular by Corollary \ref{unif-reg}. Otherwise, we have $|X_1|<|X_2|$.

\vspace{0.1cm}
Let $f\in \Gamma(X,\mathcal{P})$. By Proposition \ref{image-size-equal}, we must have  $X_1f=X_1$. If $I_2\chi^{(f)} = \{1\}$, then we are done by Theorem \ref{Gamma-regulr}. Otherwise, we have  $X_2 f = X_2$. Then each block map of $B(f, I)$ is injective and so $f$ is regular by Theorem \ref{Gamma-regulr}. Since $f$ is an arbitrary map of $\Gamma(X,\mathcal{P})$, the semigroup $\Gamma(X,\mathcal{P})$ is regular. This completes the proof.
\end{proof}

If $X$ is a finite set, the following proposition proves that every regular element of the semigroup $\Gamma(X,\mathcal{P})$ is unit-regular.
\begin{proposition}
If $\mathcal{P} = \{X_i|\; i\in I_m\}$ is an $m$-partition of a finite set $X$, then every regular element of the semigroup $\Gamma(X,\mathcal{P})$ is unit-regular.
\end{proposition}

\begin{proof}
Let $f\in \Gamma(X,\mathcal{P})$ be a regular map. Let $j\in I_m\chi^{(f)}$. Then, by Theorem \ref{Gamma-regulr}, there exists $i\in I_m$ such that the block map $f_i \colon X_i \to X_j$ of $B(f, I)$ is injective. By Remark \ref{gamma-block-surj}, we know that every block map of the family $B(f, I)$ is surjective. It follows that the block map $f_i \in B(f, I)$ is bijective and so $|X_i| = |X_j|$. Moreover, $X_i f = X_i f_i = X_j = X_j \cap Xf$. Hence $f$ is unit-regular by Theorem \ref{unit-reg-txp}.
\end{proof}

\vspace{0.02cm}
\subsection{Idempotents}
The following theorem characterizes the idempotents of the semigroup $\Gamma(X, \mathcal{P})$ in terms of certain block maps.

\begin{theorem}\label{ido-fi-iden}
Let $\mathcal{P} = \{X_i|\; i\in I\}$ be a partition of an arbitrary set $X$, and let $f\in \Gamma(X, \mathcal{P})$. Then $f$ is an idempotent if and only if for each $i\in I\chi^{(f)}$ the block map $f_i\in B(f, I)$ is the identity map.
\end{theorem}

\begin{proof}
Suppose first that $f\in \Gamma(X, \mathcal{P})$ is an idempotent. Let $i\in I \chi^{(f)}$. Since $\Gamma(X, \mathcal{P}) \subseteq T(X, \mathcal{P})$, the block map $f_i\in B(f, I)$ is an idempotent by \cite[Proposition 5.4]{shubh-c20}. By Remark \ref{gamma-block-surj}, the block map $f_i\in B(f, I)$ is a surjection. Note that any idempotent map on a set acts as the identity map on its image set (cf. \cite[p.6]{clifford61}).
Combining these, we conclude that $f_i \in B(f, I)$ is the identity map. Since $i\in I \chi^{(f)}$ is an arbitrary element, this completes the proof of the necessity part.

\vspace{0.1cm}
Conversely, suppose that the condition holds and we need to show that $f$ is an idempotent. It suffices to show that $f$ is the identity map on its image set $Xf$ (cf. \cite[p.6]{clifford61}). Let $x\in Xf$. Then there exists  $X_j \in \mathcal{P}$ such that $x\in X_j$. Clearly $j \in I \chi^{(f)}$. By hypothesis, the block map $f_j \in B(f, I)$ is the identity map. Hence  $xf = xf_j = x$. Since $x\in Xf$ is an arbitrary element, it follows that $f$ is the identity map on the image set $Xf$. This completes the proof.
\end{proof}

\section{The cardinality of $\Gamma(X, \mathcal{P})$, $E(\Gamma(X, \mathcal{P}))$, and $\mbox{Reg}(\Gamma(X, \mathcal{P}))$}
Throughout this section, $X$ is a finite set and $\mathcal{P}$ is an $(m, k)$-partition of $X$, where $m, k \in \mathbb{N}$ with $m \ge k$. Moreover, $\mathcal{P}$ has $m_i$ blocks of size $n_i$ for each $i\in I_k$ and $n_1 < n_2< \cdots < n_k$. Thus, $m = m_1 + m_2 + \ldots + m_k$.

\vspace{0.1cm}
The aim of this section is to count the number of elements, idempotents, and regular elements of the semigroup $\Gamma(X, \mathcal{P})$ for a finite set $X$ in the respective subsections. Before we calculate these, we state a remark and prove a simple lemma.

\vspace{0.1cm}
We immediately state the following from Remark \ref{gamma-block-surj}.
\begin{remark}\label{dom-ran-size}
Let $\mathcal{P} = \{X_i|\; i\in I_m\}$ be a partition of a finite set $X$, and let $f\in \Gamma(X, \mathcal{P})$. If $f_i \in B(f, I_m)$, then $|\mbox{ran}(f_i)| \le |\mbox{dom}(f_i)|$ for all $i\in I_m$.
\end{remark}

\vspace{0.01cm}
\begin{lemma}\label{size-r-subpart}
Let $\mathcal{P}$ be an $(m, k)$-partition of a finite set $X$, and let $r\in I_m$. Let $\mathfrak{F}_r$ be the collection of all $r$-subpartitions of $\mathcal{P}$ such that each $r$-subpartition in $\mathfrak{F}_r$ contains at least one block of size $n_1$. Then \[|\mathfrak{F}_r|=\sum_{l=1}^{\min\{m_1, r\}}{m_1\choose l}{m-l\choose r-l},\] where $m_1$ is the number of blocks in $\mathcal{P}$ of size $n_1$.
\end{lemma}

\begin{proof}
Let $l\in \mathbb{N}$ with $1\le l \le m_1$. Then the number of $r$-subpartitions of $\mathcal{P}$ which contain exactly $l$ blocks of the smallest size is
\[{m_1\choose l}{m-l\choose r-l}.\]
Since $l$ is arbitrary, by the addition principle, we get
\[|\mathfrak{F}_r| = \sum_{l=1}^{\min\{m_1, r\}}{m_1\choose l}{m-l\choose r-l}\]
and hence the proof is complete.
\end{proof}

\begin{notation}
For $r\in I_m$, we denote by $\mu_r$ the size of the collection $\mathfrak{F}_r$ obtained in Lemma \ref{size-r-subpart}. Moreover, we let $\mathfrak{F}_r = \{\mathcal{Q}_{r1}, \ldots, \mathcal{Q}_{r\mu_r}\}$, where each $\mathcal{Q}_{rt} \in \mathfrak{F}_r$ has $r_{t_i}\ge 0$ blocks of size $n_i$.
\end{notation}


\vspace{0.0cm}
\subsection{The cardinality of $\Gamma(X, \mathcal{P})$}

The following theorem counts the number of elements of the semigroup $\Gamma(X,\mathcal{P})$.

\begin{theorem}\label{size-gamma}
Let $\mathcal{P}$ be an $(m, k)$-partition of a finite set $X$. Then
\[|\Gamma(X,\mathcal{P})|= \prod_{i=1}^k\Big(\sum_{j=1}^im_j (n_j!) S(n_i, n_j)\Big)^{m_i}.\]
\end{theorem}

\begin{proof}
From Lemma \ref{fam-func} and Remark \ref{gamma-block-surj}, we know that each map $f \in \Gamma(X,\mathcal{P})$ is uniquely determined by the $m$-family $B(f, I_m)$ of surjective block maps. Therefore, it suffices to count the total number of such possible $m$-families $B(f,I_m)$ of surjective block maps. Since $\mathcal{P}$ has $k$ different size blocks, we break up the problem into $k$ subfamilies of surjective block maps according to their domain size .

\vspace{0.1cm}
Let $i\in I_k$. Since $\mathcal{P}$ has $m_i$ blocks of size $n_i$, we begin by counting the number of $m_i$-subfamilies of surjective block maps from $m_i$ distinct blocks of size $n_i$. By Remark \ref{dom-ran-size}, we can easily observe that the number of surjective block maps from any fixed $n_i$-element block is $\sum_{j=1}^im_j (n_j!)S(n_i,n_j)$. Therefore, by the multiplication principle, the number of possible $m_i$-subfamilies of surjective block maps from $m_i$ distinct blocks of size $n_i$ is $\Big(\sum_{j=1}^im_j (n_j!)S(n_i,n_j)\Big)^{m_i}$.

\vspace{0.1cm}
Since $i\in I_k$ is an arbitrary element, the total number of possible $m$-families of surjective block maps is
\[ \prod_{i=1}^k\Big(\sum_{j=1}^im_j (n_j!) S(n_i, n_j)\Big)^{m_i}\]
by the multiplication principle. This completes the proof.
\end{proof}

\vspace{0.2cm}
\subsection{The cardinality of $E(\Gamma(X, \mathcal{P}))$}
In this subsection, we count the number of idempotents of the semigroup $\Gamma(X, \mathcal{P})$. We begin by proving the following lemma.

\vspace{0.1cm}
\begin{lemma}\label{idem-lemma}
Let $\mathcal{P}$ be an $(m, k)$-partition of a finite set $X$, and let $\mathcal{Q}$ be an $r$-subpartition of $\mathcal{P}$ containing at least one block of size $n_1$. Let \[A_{\mathcal{Q}} = \{f \in \Gamma(X,\mathcal{P})|\; Xf  \mbox{ is the union of all blocks of }\mathcal{Q}\}.\]
Then \[|E(A_{\mathcal{Q}})|=\prod_{i=1}^k\Big( \displaystyle\sum_{j = 1}^i r_j (n_j!) S(n_i, n_j)\Big)^{m_i-r_i},\]
where $r_i\ge 0$ is the number of blocks in $\mathcal{Q}$ of size $n_i$ for each $i \in I_k$.
\end{lemma}

\begin{proof}
From Lemma \ref{fam-func} and Remark \ref{gamma-block-surj}, we know that each idempotent $f \in E(A_{\mathcal{Q}})$ is uniquely determined by the $m$-family $B(f, I_m)$ of surjective block maps. Further, from Theorem \ref{ido-fi-iden}, we know that a map $f\in A_{\mathcal{Q}}$ is idempotent if and only if each block map $f_i\in B(f, I_m)$ with $\mbox{dom}(f_i) \subseteq Xf$ is the identity map. Since $\mathcal{Q}$ is an $r$-subpartition of $\mathcal{P}$, it suffices to count the total number of such possible $(m-r)$-families $B(f, I_m)$ of surjective block maps from $(m-r)$ distinct blocks of $\mathcal{P}\setminus \mathcal{Q}$. To count it, we break up the problem into $k$ subfamilies of surjective block maps according to their domain size. Note that $r = r_1+\cdots + r_k$, $m =m_1 + \cdots + m_k$, and $m-r = (m_1 - r_1) + \cdots + (m_k - r_k)$.

\vspace{0.1cm}
Let $i\in I_k$. Since $\mathcal{P}\setminus \mathcal{Q}$ has $(m_i-r_i)$ blocks of size $n_i$, we begin by counting the number of possible $(m_i-r_i)$-subfamilies of surjective block maps from these $(m_i-r_i)$ distinct blocks of size $n_i$. By Remark \ref{dom-ran-size}, we can easily observe that the number of surjective block maps from any fixed $n_i$-element block of $\mathcal{P}\setminus \mathcal{Q}$ is $\sum_{j = 1}^i r_j (n_j!) S(n_i, n_j)$.

\vspace{0.1cm}
Recall that $\mathcal{P}\setminus \mathcal{Q}$ has $(m_i-r_i)$ blocks of size $n_i$. By the multiplication principle, the number of possible $(m_i-r_i)$-subfamilies of surjective block maps from $(m_i-r_i)$ distinct blocks in $\mathcal{P}\setminus \mathcal{Q}$ of size $n_i$ is $\Big(\sum_{j = 1}^i r_j (n_j!) S(n_i, n_j)\Big)^{m_i-r_i}$.

\vspace{0.1cm}
Since $\mathcal{P}$ has $k$ different size blocks and $i\in I_k$ is an arbitrary element, by the multiplication principle, we can get the desired number of $(m-r)$-families of surjective block maps. This completes the proof.
\end{proof}

\vspace{0.2cm}
The following theorem counts the number of elements of the set $E(\Gamma(X,\mathcal{P}))$.

\begin{theorem}
Let $\mathcal{P}$ be an $(m, k)$-partition of a finite set $X$. Then
\[|E(\Gamma(X,\mathcal{P}))|=\sum_{r=1}^m\sum_{t= 1}^{\mu_r}\prod_{i=1}^k\Big( \displaystyle\sum_{j = 1}^i r_{t_j} (n_j!) S(n_i, n_j)\Big)^{m_i-r_{t_i}},\]

where $r_{t_i}\ge 0$ is the number of blocks in $\mathcal{Q}_{rt} \in \mathfrak{F}_r$ of size $n_i$ for each $i \in I_k$.
\end{theorem}

\begin{proof}
Let $r\in I_m$, and let \[A_r = \{f \in \Gamma(X,\mathcal{P})|\; \rank\chi^{(f)} = r\}.\]
We observe that $E(A_r) = \sum E(A_{\mathcal{Q}_{rt}})$, where the sum runs over all $r$-subpartitions $\mathcal{Q}_{rt} \in \mathfrak{F}_r$ and
\[A_{\mathcal{Q}_{rt}} = \{f \in \Gamma(X,\mathcal{P})|\; Xf  \mbox{ is the union of all blocks of }\mathcal{Q}_{rt}\}.\]

\vspace{0.1cm}
Note that $|\mathfrak{F}_r| = \mu_r$. Therefore,
\begin{align*}
|E(A_r)| &= \sum_{t= 1}^{\mu_r} E(A_{\mathcal{Q}_{rt}})\\
&= \sum_{t= 1}^{\mu_r}\prod_{i=1}^k\Big( \displaystyle\sum_{j = 1}^i r_{t_j} (n_j!) S(n_i, n_j)\Big)^{m_i-r_{t_i}} \hspace{0.5cm}\mbox{ by Lemma \ref{idem-lemma}}.
\end{align*}

Since $r\in I_m$ is an arbitrary element, by the addition principle, we obtain
\[|E(\Gamma(X,\mathcal{P}))|=\sum_{r=1}^m\sum_{t= 1}^{\mu_r}\prod_{i=1}^k\Big( \displaystyle\sum_{j = 1}^i r_{t_j} (n_j!) S(n_i, n_j)\Big)^{m_i-r_{t_i}}.\]
This completes the proof.
\end{proof}

\vspace{0.2cm}

If $\mathcal{P}$ is a uniform partition of a finite set $X$, the following proposition provides a rather simple formula for the size of the set $E(\Gamma(X,\mathcal{P}))$.

\begin{proposition}
Let $\mathcal{P}$ be a uniform $m$-partition of an $n$-element set $X$. Then
\[|E(\Gamma(X,\mathcal{P}))|=\sum_{r=1}^m{m\choose r}r^{(m-r)}(q!)^{m-r},\]
where $q = \frac{n}{m}$ is the size of a block of $\mathcal{P}$.
\end{proposition}

\begin{proof}
Let $r\in I_m$, and let \[A_r = \{f \in \Gamma(X,\mathcal{P})|\; \rank\chi^{(f)} = r\}.\]
We observe that $E(A_r) = \sum E(A_{\mathcal{Q}})$, where the sum runs over all $r$-subpartitions $\mathcal{Q}$ of $\mathcal{P}$ and
\[A_{\mathcal{Q}} = \{f \in \Gamma(X,\mathcal{P})|\; Xf  \mbox{ is the union of all blocks of }\mathcal{Q}\}.\] Note that there are exactly ${m\choose r}$ choices for an $r$-subpartition $\mathcal{Q}$ of the $m$-partition $\mathcal{P}$. Therefore, by Lemma \ref{idem-lemma}, we obtain
$E(A_r) = {m\choose r} r^{m-r} (q!)^{m-r},$ where $q = \frac{n}{m}$.

\vspace{0.1cm}
Since $r\in I_m$ is an arbitrary element, by the addition principle, we get
\[|E(\Gamma(X,\mathcal{P}))| = \displaystyle\sum_{r=1}^m{m\choose r}r^{(m-r)}(q!)^{m-r},\]
where $q = \frac{n}{m}$. This completes the proof.
\end{proof}

\vspace{0.2cm}
\subsection{The cardinality of $\mbox{Reg}(\Gamma(X,\mathcal{P}))$} In this subsection, we count the number of regular elements of the semigroup $\Gamma(X,\mathcal{P})$.

\vspace{0.1cm}
If $\mathcal{P}$ is a uniform partition of a finite set $X$, we know that the semigroup $\Gamma(X, \mathcal{P})$ is regular (cf. \cite[Theorem 4.1]{pei-s05}), and so we get the size of the set $\mbox{Reg}(\Gamma(X,\mathcal{P}))$ by Theorem \ref{size-gamma}. However, note that the semigroup $\Gamma(X,\mathcal{P})$ need not be regular for an arbitrary partition $\mathcal{P}$ of $X$.

\vspace{0.2cm}

For a finite set $X$, let us recall Theorem \ref{Gamma-regulr} which can be restate immediately as follows.
\begin{remark}\label{finite-reg-gamma}
Let $\mathcal{P} = \{X_i|\;i\in I_m\}$ be a partition of a finite set $X$, and let $f\in \Gamma(X, \mathcal{P})$. Then $f$ is regular if and only if for each $j\in I_m\chi^{(f)}$, there exists $i\in I$ such that $|X_i| = |X_j|$ and $X_i f = X_j$.
\end{remark}

\vspace{0.0cm}
We now prove the following lemma.
\begin{lemma}\label{reg-lemma}
Let $\mathcal{P}$ be an $(m, k)$-partition of a finite set $X$, and let $\mathcal{Q}$ be an $r$-subpartition of $\mathcal{P}$ containing at least one block of size $n_1$. Let \[A_{\mathcal{Q}} = \{f \in \Gamma(X,\mathcal{P})|\; Xf  \mbox{ is the union of all blocks of }\mathcal{Q}\}.\] Then
\begin{equation*}
\begin{split}
|\mbox{Reg}(A_{\mathcal{Q}})|&=(r_1!)S(m_1,r_1)(n_1!)^{m_1}\\
&\quad \prod_{i=2}^k\Bigg(\sum _{p=r_i}^{m_i} \bigg({m_i\choose p}(r_i!)S(p,r_i)(n_i!)^p\Big(\sum _{j=1}^{i-1}r_j(n_j!)S(n_i,n_j)\Big)^{m_i-p}\bigg)\Bigg),
\end{split}
\end{equation*}
where $r_i \ge 0$ is the number of blocks in $\mathcal{Q}$ of size $n_i$ for each $i \in I_k$.
\end{lemma}

\begin{proof}
From Lemma \ref{fam-func} and Remark \ref{gamma-block-surj}, we know that each map $f\in \mbox{Reg}(A_{\mathcal{Q}})$ is uniquely determined by the $m$-family $B(f, I_m)$ of surjective block maps. Therefore, it suffices to count the total number of such possible $m$-families $B(f, I_m)$ of surjective block maps. Since $\mathcal{P}$ has $k$ different size blocks, we break up such $m$-families into $k$ number of $m_i$-subfamilies of surjective block maps according to their domain size.

\vspace{0.0cm}
Note that $\mathcal{P}$ has $m_1$ blocks of the smallest size $n_1$. We first count the number of $m_1$-subfamilies of surjective block maps from $m_1$ distinct blocks of size $n_1$. Note that any block in $\mathcal{P}$ of size $n_1$ can be mapped onto any block in $\mathcal{Q}$ of size $n_1$ only. Therefore, all the $m_1$ blocks in $\mathcal{P}$ of size $n_1$ can be mapped onto $r_1$ blocks in $\mathcal{Q}$ of size $n_1$ in $(r_1!)S(m_1,r_1)$ ways. Moreover, in each such way, the number of surjective block maps is $(n_1!)^{m_1}$. Therefore, the number of possible $m_1$-subfamilies of surjective block maps from $m_1$ distinct blocks of size $n_1$ is $(r_1!)S(m_1,r_1)(n_1!)^{m_1}$.

\vspace{0.2cm}
Let $i\in I_k$ be such that $i\ge 2$. Since $\mathcal{P}$ has $m_i$ blocks of size $n_i$, we now count the number of $m_i$-subfamilies of surjective block maps from $m_i$ distinct blocks of size $n_i$. Note that any block in $\mathcal{P}$ of size $n_i$ can be mapped onto any block in $\mathcal{Q}$ of size $n_j$, where $n_j \le n_i$. By Remark \ref{finite-reg-gamma}, note that there are at least $r_i$ blocks in $\mathcal{P}$ of size $n_i$ which will be mapped onto $r_i$ blocks in $\mathcal{Q}$ of size $n_i$, and the remaining at most $(m_i-r_i)$ blocks in $\mathcal{P}$ of size $n_i$ will be mapped onto blocks in $\mathcal{Q}$ of size $n_j$, where $n_j < n_i$.

\vspace{0.1cm}
Let $p$, where $r_i\leq p\leq m_i$,  be the number of blocks in $\mathcal{P}$ of size $n_i$ which map onto $r_i$ blocks in $\mathcal{Q}$ of size $n_i$. Clearly, these $p$ blocks can be mapped onto $r_i$ blocks in $(r_i !)S(p,r_i)$ ways. Moreover, in each such way, there are $(n_i!)^{p}$ surjective block maps. Note that there are ${m_i\choose p}$ ways to choose $p$ blocks among $m_i$ blocks. Therefore, by the multiplication principle, the number of such surjective block maps is ${m_i\choose p}(r_i!)S(p,r_i)(n_i!)^{p}$.

\vspace{0.1cm}
Further, each of the remaining $(m_i-p)$ blocks in $\mathcal{P}$ of size $n_i$ can be mapped onto any block in $\mathcal{Q}$ of size $n_j$, where $n_j < n_i$. Note that the number of surjective block maps from a block in $\mathcal{P}$ of size $n_i$ is $\sum_{j=1}^{i-1}r_j(n_j!)S(n_i,n_j)$. Therefore, by the multiplication principle, the number of such surjective block maps is $\big(\sum_{j=1}^{i-1}r_j(n_j!)S(n_i,n_j)\big) ^{m_i-p}$. Hence, for that fixed $p$, the number of $m_i$-subfamilies of surjective block maps from $m_i$ distinct blocks of size $n_i$ is
\[\bigg({m_i\choose p}(r_i!)S(p,r_i)(n_i!)^{p}\Big(\sum _{j=1}^{i-1}r_j(n_j!)S(n_i,n_j)\Big) ^{m_i-p}\bigg)\] by the multiplication principle.
Since $p$ runs from $r_i$ to $m_i$, summing over all admissible $p$, the total number of possible $m_i$-subfamilies of surjective block maps from $m_i$ distinct blocks of size $n_i$ is
 \[\sum_{p=r_i}^{m_i}\bigg({m_i\choose p}(r_i!)S(p,r_i)(n_i!)^{p}\Big(\sum _{j=1}^{i-1}r_j(n_j!)S(n_i,n_j)\Big) ^{m_i-p}\bigg).\]
Since $i\in I_k$ is an arbitrary element greater than $1$, by the multiplication principle, we obtain the total number of possible $m$-families of surjective block maps. This completes the proof.
\end{proof}


The following theorem counts the number of regular elements of the semigroup $\Gamma(X, \mathcal{P})$.
\begin{theorem}
Let $\mathcal{P}$ be an $(m, k)$-partition of a finite set $X$. Then
\begin{align*}
|\mbox{Reg} (\Gamma(X,\mathcal{P}))| &= \sum_{r=1}^m\sum_{t= 1}^{\mu_r}\Bigg((r_{t_1}!)S(m_1,r_{t_1})(n_1!)^{m_1}\\
&\quad \prod_{i=2}^k\Bigg(\sum _{p=r_{t_i}}^{m_i} \bigg({m_i\choose p}(r_{t_i}!)S(p,r_{t_i})(n_i!)^p\Big(\sum _{j=1}^{i-1}r_{t_j}(n_j!)S(n_i,n_j)\Big)^{m_i-p}\bigg)\Bigg)\Bigg),
\end{align*}
where $r_{t_i}\ge 0$ is the number of blocks in $\mathcal{Q}_{rt} \in \mathfrak{F}_r$ of size $n_i$ for each $i \in I_k$.
\end{theorem}

\begin{proof}
Let $r\in I_m$, and let \[A_r = \{f \in \Gamma(X,\mathcal{P})|\; \rank\chi^{(f)} = r\}.\]
We observe that $\mbox{Reg}(A_r) = \sum \mbox{Reg}(A_{\mathcal{Q}_{rt}})$, where the sum runs over all $r$-subpartitions $\mathcal{Q}_{rt} \in \mathfrak{F}_r$ and \[A_{\mathcal{Q}_{rt}} = \{f \in \Gamma(X,\mathcal{P})|\; Xf  \mbox{ is the union of all blocks of }\mathcal{Q}_{rt}\}.\]

\vspace{0.1cm}
Note that $|\mathfrak{F}_r| = \mu_r$. Therefore,
\begin{align*}
|\mbox{Reg}(A_r)| &= \sum_{t= 1}^{\mu_r} \mbox{Reg}(A_{\mathcal{Q}_{rt}})\\
&= \sum_{t= 1}^{\mu_r}\Bigg((r_{t_1}!)S(m_1,r_{t_1})(n_1!)^{m_1}\\
&\quad \prod_{i=2}^k\Bigg(\sum _{p=r_{t_i}}^{m_i} \bigg({m_i\choose p}(r_{t_i}!)S(p,r_{t_i})(n_i!)^p\Big(\sum _{j=1}^{i-1}r_{t_j}(n_j!)S(n_i,n_j)\Big)^{m_i-p}\bigg)\Bigg)\Bigg)
\end{align*}
by Lemma \ref{reg-lemma}. Since $r\in I_m$ is an arbitrary element, by the addition principle, we obtain the desired formula of $|\mbox{Reg}(\Gamma(X,\mathcal{P}))|$. This completes the proof.
\end{proof}





\end{document}